\documentclass[11pt]{article}

\usepackage{enumerate, amsmath, amsthm, amsfonts, amssymb, youngtab}

\usepackage[margin=1in]{geometry}
\usepackage{hyperref}

\newtheorem*{theorem}{Theorem}
\newtheorem*{example}{Example}

\newcommand{\Q}{\mathbf{Q}}
\newcommand{\Z}{\mathbf{Z}}
\renewcommand{\phi}{\varphi}
\DeclareMathOperator{\coker}{coker}
\newcommand{\GL}{\mathbf{GL}}
\DeclareMathOperator{\Sym}{Sym}
\newcommand{\Sc}{\mathbf{S}}
\renewcommand{\SS}{\mathfrak{S}}

\title{Computing inclusions of Schur modules}
\author{Steven V Sam}
\date{July 29, 2009}

\begin{document}
\maketitle

\begin{abstract}
  We describe a software package for constructing minimal free
  resolutions of $\GL_n(\Q)$-equivariant graded modules $M$ over
  $\Q[x_1, \dots, x_n]$ such that for all $i$, the $i$th syzygy module
  of $M$ is generated in a single degree. We do so by describing some
  algorithms for manipulating polynomial representations of the
  general linear group $\GL_n(\Q)$ following ideas of Olver and
  Eisenbud--Fl\o ystad--Weyman.
\end{abstract}

\section{Introduction.}

This article describes the Macaulay 2 package {\tt
  PieriMaps}\footnote{This article describes version 1.0 of {\tt
    PieriMaps} written July 3, 2009. As of the writing of this
  article, the latest version of Macaulay 2 (version 1.2) contains
  version 0.5 of {\tt PieriMaps}. The updated version of {\tt
    PieriMaps} can be downloaded at
  \url{http://math.mit.edu/~ssam/PieriMaps.m2}.}, which defines maps
of representations of the general linear group $\GL_n(\Q)$ of the form
\[
\Sc_\mu(\Q^n) \to \Sc_{(d)}(\Q^n) \otimes \Sc_\lambda(\Q^n),
\]
and presents them as degree 0 maps
\[
A \otimes \Sc_\mu(\Q^n)(-d) \to A \otimes \Sc_\lambda(\Q^n)
\]
of free modules over the polynomial ring $A = \Sym(\Q^n)$. Here
dominant weights of $\GL_n(\Q)$ are identified with weakly decreasing
sequences $\lambda$ of length $n$, and $\Sc_\lambda(\Q^n)$ denotes the
irreducible representation of highest weight $\lambda$. Such maps are
of general importance, and appeared recently in the work of Eisenbud,
Fl\o ystad and Weyman \cite{efw}. The package also describes certain
related maps in characteristic $p$.

We give some context for this work and then describe the contents of
this article.

Let $K$ be a field, and $A = K[x_1, \dots, x_n]$ the polynomial ring
in $n$ variables. In a recent paper of Eisenbud and Schreyer
\cite{eisenbudschreyer}, a theorem regarding the ``shape'' of minimal
free resolutions of Cohen--Macaulay $A$-modules was established. Given
a Cohen--Macaulay $A$-module $M$, the {\bf Betti diagram}
$\beta_{i,j}(M)$ is the number of generators of degree $j$ in the
$i$th syzygy module of a minimal free resolution of $M$. A Betti
diagram is {\bf pure} if for each $i$, $\beta_{i,j}(M) \ne 0$ for at
most one $j$. In this case, we let $(d_1, \dots, d_r)$ be the degree
sequence of $\beta$: that is, $\beta_{i,d_i}(M) \ne 0$ for all
$i$. The theorem mentioned above states that any Betti diagram of a
Cohen--Macaulay module is a rational linear combination of pure Betti
diagrams. The Herzog--K\"uhl equations \cite[Theorem 1]{herzogkuhl}
show that each strictly increasing degree sequence determines the
corresponding Betti diagram up to a rational multiple. For
$\operatorname{char} K = 0$, Eisenbud, Fl\o ystad, and Weyman
\cite{efw} constructed pure free resolutions for each degree sequence
which live in the category of $\GL_n(K)$ representations.  Eisenbud
and Schreyer give a characteristic free construction which gives
different rational multiples of the Betti diagrams in general. It is
still not completely known which multiples can and cannot come from
the Betti diagram of a graded module. For example, it is an
interesting open problem to determine for a given degree sequence the
smallest integer multiple given by the Herzog--K\"uhl equations which
actually comes from a module.

It is the goal of this article to describe how these resolutions can
be represented concretely in Macaulay~2 \cite{m2}. One can find
$\Z$-forms for these maps and work in positive characteristic, and one
such $\Z$-form is implemented, but in general it will not produce pure
resolutions. We should mention that this choice of $\Z$-form is not
unique. It would be interesting to investigate how often the
characteristic $p$ resolutions will be pure, and to construct
$\Z$-forms which give equivariant pure resolutions in positive
characteristic.

The rest of the article is organized as follows. In
Section~\ref{representations}, we review a construction for
representations of $\GL_n(\Q)$. In Section~\ref{purefreesection}, we
give the construction of Eisenbud, Fl\o ystad, and Weyman, and its
extension to characteristic $p$. In
Section~\ref{combinatorialsection}, we describe the differentials in
terms of bases, in the way that it is implemented in {\tt PieriMaps},
and illustrate an example. Finally, in Section~\ref{m2code} we give
some examples of Macaulay 2 code which show how one can use this
package. 

\section{Representations of $\GL_n(\Q)$.} \label{representations}

In this section, we present a construction for irreducible polynomial
representations of the rational algebraic group $\GL_n(\Q)$ which is
convenient for our purposes.

Let $\lambda = (\lambda_1, \dots, \lambda_n)$ ($\lambda_1 \ge \cdots
\ge \lambda_n \ge 0$) be a partition and $m = |\lambda| = \lambda_1 +
\cdots + \lambda_n$. The {\bf Young diagram} of $\lambda$ is a
pictorial representation of $\lambda$: we draw $n$ rows (some may be
empty) of boxes with $\lambda_i$ boxes in the $i$th row, making sure
that each row is left-justified. The notation $(i,j)$ refers to the
box in the $i$th row and $j$th column. A {\bf filling} of shape
$\lambda$ is an assignment of the numbers $\{1,\dots,n\}$ (repetitions
allowed) to the boxes of the Young diagram of $\lambda$. By picking
some order on the boxes of $\lambda$, we get an action of the
symmetric group $\SS_m$ on the fillings of $\lambda$. We'll say that
$\sigma \in \SS_m$ is {\bf row-preserving} if it permutes the rows of
$\lambda$ amongst themselves. The {\bf Schur
  module}\footnote{Actually, we are defining the {\bf Weyl module} of
  highest weight $\lambda$, but in characteristic 0, it is isomorphic
  to what is usually called the Schur module.}  $\Sc_\lambda(\Q^n)$ is
the rational vector space with basis given by the fillings $T$ of
$\lambda$ together with the following relations:
\begin{enumerate}
\item (Symmetric relation) $T = \sigma \cdot T$ for any row-preserving
  permutation $\sigma$.
\item (Shuffle relation) For $i$ and $j$ such that $(i,j)$ and
  $(i+1,j)$ are boxes of $\lambda$, let $B = \{(i,k) \mid j \le k \le
  \lambda_i \} \cup \{(i+1,k) \mid 1 \le k \le j\}$. Then $\sum \sigma
  \cdot T = 0$, where the sum is over all permutations which fix all
  boxes not in $B$.
\end{enumerate}
For more details, the reader is referred to \cite[Proposition
2.1.15]{weyman} where our notion of Schur module is called a Weyl
functor, and is denoted by $K_\lambda$. The above presentation
implicitly replaces the use of divided powers with symmetric powers,
but this distinction is irrelevant in characteristic 0.

A filling is a {\bf semistandard tableau} if the numbers are weakly
increasing from left to right along rows, and strictly increasing from
top to bottom along columns. A basis for $\Sc_\lambda(\Q^n)$ (over
$\Q$) is given by the semistandard tableaux of shape $\lambda$.  We
define an action of $\GL_n(\Q)$ on $\Sc_\lambda(\Q^n)$ as
follows. Given a filling $T$, let $(j_1, \dots, j_m)$ be its entries
(in some order). For $g = (g_{i,j}) \in \GL_n(\Q)$, set $g \cdot T =
\sum_I g_{i_1,j_1} \cdots g_{i_m, j_m} T_I$ where the sum is over all
index sets $I = (i_1, \dots, i_m) \in \{1,\dots,n\}^m$, and $T_I$ is
the filling obtained by replacing each $j_k$ by $i_k$.

We will need {\bf Pieri's formula}: if $(d)$ is a partition with one
row, then 
\[
\Sc_{(d)}(\Q^n) \otimes_\Q \Sc_\lambda(\Q^n) \cong \bigoplus_\mu
\Sc_\mu(\Q^n)
\]
where $\mu$ ranges over all partitions with at most $n$ parts obtained
from $\lambda$ by adding $d$ boxes, no two of which are in the same
column. Thus there are inclusions (unique up to scalar multiple)
\[
\Sc_\mu(\Q^n) \to \Sc_{(d)}(\Q^n) \otimes_\Q \Sc_\lambda(\Q^n),
\]
which we will call {\bf Pieri inclusions}. We remark that this direct
sum decomposition is only valid in characteristic 0, and is the main
barrier to extending the setup of this article to positive
characteristic.

\section{Equivariant pure free resolutions in characteristic
  0.} \label{purefreesection}

Fix a degree sequence $d = (d_0, \dots, d_n)$. Define a partition
$\alpha(d,0) = \lambda$ by $\lambda_i = d_n - d_i - n + i$, and for $1
\le j \le n$, define partitions
\[
\alpha(d,j) = (\lambda_1 + d_1 - d_0, \lambda_2 + d_2 - d_1, \dots,
\lambda_j + d_j - d_{j-1}, \lambda_{j+1}, \lambda_{j+2}, \dots,
\lambda_n).
\]
Let $A = \Q[x_1, \dots, x_n]$, and define $A$-modules ${\bf F}(d)_i$
for $0 \le i \le n$ by
\[ {\bf F}(d)_i = A(-d_i) \otimes_\Q \Sc_{\alpha(d,i)}(\Q^n)
\]
(Here $A(a)$ denotes a grading shift by $a$.) The natural action of
$\GL_n(\Q)$ on $A = \bigoplus_{i \ge 0} \Sc_{(i)}(\Q^n)$ and on
$\Sc_{\alpha(d,i)}(\Q^n)$ gives an action of $\GL_n(\Q)$ on ${\bf
  F}(d)_i$. Note that $|\alpha(d,i)| - |\alpha(d,i-1)| = d_i -
d_{i-1}$, and that $\alpha(d,i)$ is obtained from $\alpha(d,i-1)$ by
adding boxes only in the $i$th row, so there exists a Pieri inclusion
\[ 
\phi_i \colon \Sc_{\alpha(d,i)}(\Q^n) \to \Sc_{(d_i - d_{i-1})}(\Q^n)
\otimes_\Q \Sc_{\alpha(d,i-1)}(\Q^n)
\]
Identifying $\Sc_{(d_i - d_{i-1})}(\Q^n) = \Sym^{d_i - d_{i-1}}(\Q^n)$
gives a degree 0 map $\partial_i \colon {\bf F}(d)_i \to {\bf
  F}(d)_{i-1}$ given by $p(x) \otimes v \mapsto p(x)\phi_i(v)$.

\begin{theorem}[Eisenbud--Fl{\o}ystad--Weyman] \label{equivariantres}
  With the notation above,
  \[
  0 \to {\bf F}(d)_n \xrightarrow{\partial_n} \cdots
  \xrightarrow{\partial_2} {\bf F}(d)_1 \xrightarrow{\partial_1} {\bf
    F}(d)_0
  \]
  is a $\GL_n(\Q)$-equivariant minimal graded free resolution of $M(d)
  = \coker \partial_1$, which is pure of degree $d$. Furthermore,
  $M(d)$ is isomorphic, as a $\GL_n(\Q)$ representation, to the direct
  sum of all irreducible summands of $A \otimes_\Q \Sc_\lambda(\Q^n)$
  corresponding to the partitions that do not contain $\alpha(d,1)$,
  and in particular is a module of finite length.
\end{theorem}

\begin{proof} See \cite[Theorem 3.2]{efw}. We should emphasize that
  our notation for partitions differs from that of (loc. cit.) in that
  the notions of rows and columns are interchanged. \end{proof}

An implementation of the map $\partial_1$ is given in the method {\tt
  pureFree} in {\tt PieriMaps}. One can then compute the remaining
maps and compose them, or compute a minimal free resolution using
Macaulay 2.

The extension of this construction to characteristic $p$ in general
does not produce pure resolutions. The idea is to clear the
denominators in the matrix giving a Pieri inclusion, remove the
torsion from the cokernel, and then reduce coefficients modulo $p$,
for a given prime. This is also implemented in the method {\tt
  pureFree}: the user need only specify a characteristic.

\section{Combinatorial description of the Pieri
  inclusion.} \label{combinatorialsection}

We wish to describe how the Pieri inclusions work in terms of the
bases of semistandard tableaux of the Schur modules
$\Sc_\lambda(\Q^n)$. The Pieri inclusion which induces the map
$\partial_i \colon {\bf F}(d)_i \to {\bf F}(d)_{i-1}$ is given by
\[
\Sc_{\alpha(d,i)}(\Q^n) \to \Sc_{(d_i - d_{i-1})}(\Q^n) \otimes_\Q
\Sc_{\alpha(d,i-1)}(\Q^n),
\]
and $\alpha(d,i-1)$ is obtained from $\alpha(d,i)$ by removing $d_i -
d_{i-1}$ boxes from the $i$th row. So to describe this map, we first
describe the case $d_i - d_{i-1} = 1$. In this case, the map was
described by Olver in \cite[\S 6]{olver}.  For the general case, one
can iterate this process of removing one box at a time and compose the
maps. It needs to be proved (though it is not hard), that such a
composition is the desired map, and in fact, if one removes boxes from
multiple rows, the order in which the boxes are removed is irrelevant
(up to nonzero scalar multiple).

Suppose we have a partition $\lambda$ with $\lambda_{k-1} >
\lambda_k$. Let $\mu$ be the partition resulting from adding a box to
the $k$th row of $\lambda$. Set $B_k = \{ (j_1, \dots, j_p) \mid 0 =
j_1 < \cdots < j_p = k\}$, and for $J = (j_1, \dots, j_p) \in B_k$,
define $\#J = p$. Given $x^a \otimes T$ where $x^a = x_1^{a_1} \cdots
x_n^{a_n} \in \Q[x_1, \dots, x_n]$ is a monomial and $T$ is a filling
of $\mu$, we first interpret the $x^a$ as a ``zeroth'' row of $T$
consisting of $a_1 + \cdots + a_n$ boxes filled with $a_i$
$i$'s. We'll call this a {\bf shape}. Given numbers $0 \le i < j \le
k$, define $\tau_{i,j}(x^a \otimes T)$ to be the sum of all shapes
obtained from $x^a \otimes T$ by removing a box along with its entry
(and then the boxes to the right of it get shifted one to the left)
from row $j$ and moving it to the end of row $i$. Then set
\begin{align} \label{denominators}
\tau_J = \tau_{j_{p-1},j_p} \circ \cdots \circ \tau_{j_2, j_3} \circ
\tau_{j_1, j_2}, \quad c_J = \prod_{i=2}^{p-1} (\mu_{j_i} - \mu_k + k
- j_i).
\end{align}
The desired map $\Sc_\mu(\Q^n) \to \Sc_{(1)}(\Q^n) \otimes_\Q
\Sc_\lambda(\Q^n)$ is now the alternating sum $\displaystyle \sum_{J
  \in B_k} \frac{(-1)^{\#J} \tau_J}{c_J}$. We give an example to
illustrate all of the above.

\begin{example} \label{combinpieri} \rm Let $\mu = (2,1,1)$, $k=3$,
  $n=3$, and consider the element $1 \otimes T$ where $T$ is the
  semistandard tableau
\[
\Yvcentermath1 T = \young(12,2,3).
\]
We think of $T$ as having a ``row 0'' which is an empty row on top of
$T$. If $J = (0,1,3)$, then 
\[
\Yvcentermath1 \tau_J(T) = \tau_{1,3}(\tau_{0,1}(T)) = 
\tau_{1,3}\left(\ \young(1,2,2,3) + \young(2,1,2,3)\ \right)
= \young(1,23,2) + \young(2,13,2)
\]
which we really think of as
\[
\Yvcentermath1 x_1 \otimes \young(23,2) +
x_2 \otimes \young(13,2) =
-\frac{1}{2} x_1 \otimes \young(22,3) +
x_2 \otimes \young(13,2)
\]
in the module $\Q[x_1, x_2, x_3] \otimes_\Q \Sc_{(2,1)} V$. The
equality follows from the relations described in
Section~\ref{representations}. In this case, $c_J = 2$.
\end{example}

\section{An example of using {\tt PieriMaps}.} \label{m2code}

We illustrate some of the main uses of {\tt PieriMaps}. First, we load
the package and define a polynomial ring in 3 variables $A =
\Q[a,b,c]$:
\begin{verbatim}
i1 : loadPackage "PieriMaps"
i2 : A = QQ[a,b,c];
\end{verbatim}
Now we compute a module whose pure free resolution has degree sequence
$\{0,1,3,5\}$.
\begin{verbatim}
i3 : pureFree({0,1,3,5}, A)

o3 = | 3a 0  b  0  c  0  0  0  0     0 0   0  0  0  0  |
     | 0  3a 0  b  0  c  0  0  0     0 0   0  0  0  0  |
     | 0  0  2a 0  0  0  2b 0  c     0 0   0  0  0  0  |
     | 0  0  0  2a 0  0  0  2b 0     c 0   0  0  0  0  |
     | 0  0  0  0  2a 0  0  0  b     0 2c  0  0  0  0  |
     | 0  0  0  0  0  2a 0  0  0     b 0   2c 0  0  0  |
     | 0  0  0  0  0  0  0  a  -1/2a 0 0   0  3b c  0  |
     | 0  0  0  0  0  0  0  0  0     a -2a 0  0  2b 2c |

             8       15
o3 : Matrix A  <--- A
\end{verbatim}
This is the matrix of free $A$-modules induced by the Pieri inclusion
$\Sc_{3,1}(\Q^3) \to \Sc_1(\Q^3) \otimes \Sc_{2,1}(\Q^3)$. The bases
of $A^{15}$ and $A^8$ can be listed with the commands {\tt
  standardTableaux(3, \{3,1\})} and {\tt standardTableaux(3,
  \{2,1\})}, respectively. For example, the first command has {\tt
  \{\{0,0,0\}, \{1\}\}} as its first basis element, which is meant to
represent the semistandard tableau $\Yvcentermath1
\young(000,1)$. Alternatively, this map can be produced with the
command {\tt pieri(\{3,1,0\}, \{1\}, A)} because the partition
$(2,1,0)$ is obtained by subtracting 1 from the first entry of
$(3,1,0)$. The module we are after is the cokernel of this map.
\begin{verbatim}
i4 : res coker oo

      8      15      10      3
o4 = A  <-- A   <-- A   <-- A  <-- 0
                                    
     0      1       2       3      4

o4 : ChainComplex
\end{verbatim}
We can check that this resolution is pure by looking at its Betti
table:
\begin{verbatim}
i5 : betti oo

            0  1  2 3
o5 = total: 8 15 10 3
         0: 8 15  . .
         1: .  . 10 .
         2: .  .  . 3

o5 : BettiTally
\end{verbatim}
We can lift the above map to a $\Z$-form and reduce the coefficients
modulo 2:
\begin{verbatim}
i6 : pieri({3,1},{1},ZZ/2[x,y,z])

o6 = | x 0 y 0 0 z 0 0 0 0 0 0 0 0 0 |
     | 0 x 0 0 y 0 0 z 0 0 0 0 0 0 0 |
     | 0 0 0 y 0 0 z 0 0 0 0 0 0 0 0 |
     | 0 0 0 0 0 0 y 0 z 0 0 0 z 0 0 |
     | 0 0 0 0 0 0 y 0 z 0 0 0 0 0 0 |
     | 0 0 0 0 0 0 0 0 y z 0 0 y 0 0 |
     | 0 0 0 0 0 0 0 0 0 0 x y 0 z 0 |
     | 0 0 0 0 0 0 0 0 0 0 0 0 x 0 z |

             ZZ          8       ZZ          15
o6 : Matrix (--[x, y, z])  <--- (--[x, y, z])
              2                   2
\end{verbatim}
However, the resolution of its cokernel is not pure:
\begin{verbatim}
i7 : betti res coker oo

            0  1  2 3
o7 = total: 8 15 10 3
         0: 8 15  9 3
         1: .  .  . .
         2: .  .  1 .

o7 : BettiTally
\end{verbatim}
Now let's look at an example of changing the order of composition of
Pieri inclusions. We know that there is a nonzero inclusion of the
form $\Sc_{2,1}(\Q^3) \to \Sc_2(\Q^3) \otimes \Sc_1(\Q^3)$. There are
two different ways to get this map with the function {\tt pieri}. We
could remove a box from the second row of $(2,1,0)$ and then remove a
box from the first row of $(2,0,0)$ to get the composition
\[
\Sc_{2,1}(\Q^3) \to \Sc_1(\Q^3) \otimes \Sc_2(\Q^3) \xrightarrow{1
  \otimes \phi} \Sc_1(\Q^3) \otimes \Sc_1(\Q^3) \otimes \Sc_1(\Q^3)
\xrightarrow{p \otimes 1} \Sc_2(\Q^3) \otimes \Sc_1(\Q^3),
\]
where $\phi$ is a Pieri inclusion and $p$ is the quotient map:
\begin{verbatim}
i8 : pieri({2,1}, {2,1}, A)

o8 = | ab  ac  1/2b2  1/2bc 1/2bc 1/2c2  0   0      |
     | -a2 0   -1/2ab 1/2ac -ac   0      bc  1/2c2  |
     | 0   -a2 0      -ab   1/2ab -1/2ac -b2 -1/2bc |

             3       8
o8 : Matrix A  <--- A
\end{verbatim}
Or, we could remove a box from the first row of $(2,1,0)$ and then
remove a box from the second row of $(1,1,0)$ to get the composition 
\[
\Sc_{2,1}(\Q^3) \to \Sc_1(\Q^3) \otimes \Sc_{1,1}(\Q^3) \xrightarrow{1
  \otimes \phi} \Sc_1(\Q^3) \otimes \Sc_1(\Q^3) \otimes \Sc_1(\Q^3)
\xrightarrow{p \otimes 1} \Sc_2(\Q^3) \otimes \Sc_1(\Q^3),
\]
which is written as
\begin{verbatim}
i9 : pieri({2,1}, {1,2}, A)

o9 = | 2ab  2ac  b2  bc   bc   c2  0    0   |
     | -2a2 0    -ab ac   -2ac 0   2bc  c2  |
     | 0    -2a2 0   -2ab ab   -ac -2b2 -bc |

             3       8
o9 : Matrix A  <--- A
\end{verbatim}
Here, we see that the matrices differ by a scalar multiple of 2. In
general, different orders of box removals will yield the same matrix
up to nonzero scalar multiple. The differences arise from the
denominators $c_J$ (see \eqref{denominators}).

\paragraph{Acknowledgements.} The author thanks David Eisenbud and
Jerzy Weyman for helpful comments and encouragement while the package
{\tt PieriMaps} was written, and for reading a draft of this
article. The author also thanks an anonymous referee for suggesting
some improvements.

\bigskip

\noindent Steven V Sam\\
Department of Mathematics\\
Massachusetts Institute of Technology\\
Cambridge, MA 02139\\
{\tt ssam@math.mit.edu}\\
\url{http://math.mit.edu/~ssam/}

\end{document}